\documentclass[preprint,12pt]{elsarticle}
\usepackage{hyperref}
\usepackage{graphicx}
\usepackage{listings}
\usepackage{amsmath, amssymb, amsfonts, latexsym,amsthm}
\usepackage[mathscr]{euscript}
\usepackage{wasysym}
\usepackage{graphics}
\usepackage{enumerate}
\usepackage{enumitem}
\usepackage[normalem]{ulem}
\usepackage{graphicx}
\usepackage{xcolor}
\usepackage{bm}
\usepackage{listings}
\usepackage{tikz}
\usetikzlibrary{arrows.meta, bending, decorations.pathmorphing, decorations.pathreplacing, decorations.shapes}
\usepackage{algorithm}
\usepackage[noend]{algpseudocode}

\newtheorem{theorem}{Theorem}[section]
\newtheorem{lemma}[theorem]{Lemma}
\newtheorem{corollary}[theorem]{Corollary}
\newtheorem{definition}{Definition}[section]
\newtheorem{proposition}[theorem]{Proposition}

\algrenewcommand\algorithmicrequire{\textbf{Input:}}
\algrenewcommand\algorithmicensure{\textbf{Output:}}
\newcommand{\dG}{\overrightarrow{G}}

\DeclareMathOperator{\lca}{lca}
\DeclareMathOperator{\qBMG}{qBMG}
\newcommand{\AX}[1]{\textnormal{#1}}
\journal{Discrete Applied Mathematics}

\begin{document}

\begin{frontmatter}
\title{Forbidden configurations and dominating bicliques in undirected 2-quasi best match graphs}

\author[1]{Annachiara Korchmaros\corref{cor1}}
\ead{annachiara.korchmaros@uni-leipzig.de}
\cortext[cor1]{Annachiara Korchmaros}
\author[1,2,3,4,5]{Peter F. Stadler} 
\affiliation[1]{organization={Universität Leipzig},
            addressline={Härtelstr. 16-18},
             city={Leipzig},
             postcode={D-04107},
             country={Germany}}
\affiliation[2]{organization={Max Planck Institute for
  Mathematics in the Sciences},
            addressline={Inselstraße 22}, 
            city={Leipzig},
            postcode={ 04103}, 
            country={Germany}}
\affiliation[3]{organization={University of Vienna},
            addressline={Universitätsring 1}, 
            city={Vienna},
            postcode={1010}, 
            country={Austria}}
\affiliation[4]{organization={Santa Fe Institute},
            addressline={1399 Hyde Park Rd}, 
            city={Santa Fe},
            postcode={87501,}, 
            state={New Mexico},
            country={USA}}
\affiliation[5]{organization={Universidad Nacional de Colombia},
            addressline={Ave Cra 30}, 
            city={Bogotá},
            country={Colombia}}

\begin{abstract}
2-quasi best match graphs (2qBMGs) are directed graphs that capture a
  notion of close relatedness in phylogenetics. Here, we investigate the underlying undirected graph of a 2qBMG (un2qBMG) and show that it contains neither a path $P_l$ nor a cycle $C_l$ of length $l\geq 6$ as an induced subgraph. This property guarantees the existence of specific vertex decompositions with dominating bicliques that provide further insights into their structure.
\end{abstract}


\end{frontmatter}
\section{Introduction}
\label{sec:intro}
Relationships between genes that are of key interest in evolutionary biology give rise to vertex colored graphs that are defined in terms of phylogenetic trees that model evolutionary history. This binary relationship can be inferred from data, at least in principle, prompting
the development of a method to infer the trees from the relational data; see
e.g.~\cite{ramirez2024revolutionh,schaller2021corrigendum,korchmaros2021structure,korchmaros2023quasi}.
Best Match Graphs (BMGs) are a starting point for inferring pairs of orthologous genes; see Section~\ref{sec:background} for a formal definition of a BMG. The basic idea is to capture pairs of genes $a$
and $b$ from species $A$ and $B$, respectively, which are evolutionarily
most closely related.  BMGs are directed graphs with a vertex coloring
$\sigma$ assigning to each vertex (gene) the species in which it resides.
Several features of BMGs can be phrased in terms of phylogenetic
trees; see~\cite{schaller2021corrigendum,schaller2021complexity}.  In
particular, certain three-vertex subgraphs in BMGs give rise to binary
trees on three leaves that must or must not be displayed by any tree that
explains these BMGs. This and additional properties yield a classical
characterization of BMGs~\cite[Theorem 9]{schaller2021corrigendum}.
Equivalently, BMGs are precisely the so-called color-sink-free quasi-best
match graphs (qBMGs)~\cite[Theorem~4.3]{korchmaros2023quasi}, i.e. no vertex has an empty out-neighborhood.  In particular, quasi-best match
graphs are a natural generalization of BMGs that allows us to keep only
those best matches that are evolutionary, not too far in the past~\cite{korchmaros2023quasi}. Two-colored qBMGs (2qBMG) are a rather special subclass of qBMGs. They are relevant in practice because every subgraph of a qBMG induced by a pair of colors is a 2qBMG.
While all known characterizations of qBMGs explicitly involve the tree
structure on which the best matches are defined, tree-free
characterizations have been derived for the two-colored case: 
\begin{proposition} \cite{schaller2021corrigendum,korchmaros2023quasi}
\label{pro23112025}
  A directed graph $(\dG,\sigma)$ with a proper two-coloring $\sigma$
    of its vertex set is a 2qBMG if and only if the in-neighborhoods
  $N^-(v)$, $N^-(u)$ and out-neighborhoods $N^+(v)$, $N^+(u)$ of every two
  distinct vertices $u$ and $v$ of $\dG$ satisfy the following three
  conditions:
  \begin{itemize} \label{Ns}
  \item[\texorpdfstring{\AX{(N1)}}{(N1)}] $u\notin N^+(v),\,u\notin N^+(v)$ implies
    $N^+(u)\cap N^+(N^+(v))=N^+(v)\cap N^+(N^+(u))=\emptyset$.
  \item[\texorpdfstring{\AX{(N2)}}{(N2)}] $N^+(N^+(N^+(u))) \subseteq N^+(u)$.
  \item[\texorpdfstring{\AX{(N3)}}{(N3)}] If $N^+(u)\cap N^+(v)\ne\emptyset$ then
    $N^+(u)\subseteq N^+(v)$ or
    $N^+(v)\subseteq N^+(u)$.
  \end{itemize}
\end{proposition}
Properties (N1) and (N2) may be rephrased as
in~\cite{korchmaros2021structure}:
\begin{itemize} 
\item[(N1)] A (N1)-\emph{configuration} $[u,t,w,v]$ consists of four pairwise distinct vertices $u,t,w,v$ such that $u,v$ are independent but $ut$, $vw$, $tw$ are edges. Condition (N1) means that $(\dG,\sigma)$ contains no (N1)-configuration.
\item[(N2)] If $uv$, $vw$, and $wt$ are edges, then $ut$ is also an edge.
\end{itemize}
Directed graphs satisfying (N2) are sometimes called bi-transitive.

The 2qBMGs form a class of bipartite properly vertex-colored directed graphs, since every digraph with both properties (N1) and (N2) is bipartite; see~\cite[Theorem 7.9]{korchmaros2023quasi}.
  
The subgraphs of a BMG comprising all its pairs of symmetric edges (i.e., 2-cycles) are the reciprocal BMGs, which are of key interest in studying gene family histories~\cite{hellmuth2020complexity}. 
  From a mathematical point of view, it is a natural idea to consider the
underlying undirected graphs of BMGs and qBMGs, aiming to exploit profound
results on undirected graphs to gain new insights into 2qBMGs. Here, we are
particularly interested in forbidden configurations and edge
decompositions, as well as in algorithms and computational complexity. However, we note that the relationships between the classes of digraphs
$\overrightarrow{G}$ and their underlying undirected graph $G$ are far from
obvious. For example, 2qBMGs and un2qBMGs form hereditary classes of
graphs; see Proposition~\ref{prop:un2qbmg_hereditary} and
\cite{korchmaros2023quasi}. However, the characterization of 2qBMGs by the
forbidden induced subgraphs,
see~\cite[Theorem~4.4]{schaller2021complexity}, does not provide an
analogous characterization of un2qBMGs. The reason is that the underlying
undirected graph of a forbidden induced subgraph of 2qBMGs is not
necessarily a forbidden subgraph of the undirected version. Moreover,
  if there are no vertices with the same in-neighbors and out-neighbors,
  2qBMGs may have a large symmetry group, while 2BMGs only have a few
  symmetries; see~\cite{KADAM}. The relationships between
$\overrightarrow{G}$ and $G$ are thus far from trivial.

Our main results in this direction are Theorems~\ref{thm:P6free} and
\ref{thm:odd-even}. The first result, which establishes that $P_6$ and
$C_6$ (and thus also all cycles of length $>6$) are forbidden induced
subgraphs for un2qBMGs, was announced in the extended
abstract~\cite{AK2024ITAT}. This result is sharp in the sense that a
un2qBMG may have an induced path $P_5$ (and hence induced $P_4$ as well);
see Propositions~\ref{prop:P5} and~\ref{thm:P4-cases}. A graph is
\emph{chordal bipartite} if it is bipartite and does not contain an induced
cycle of length at least six. Un2qBMGs are chordal bipartite
graphs. However, the converse is not true; see
Corollary~\ref{cor:strictly_P6_C6}.  Recognition and optimization problems
for $P_6$-free chordal bipartite graphs, $(P_6,C_6)$-free bipartite graphs
have been studied by several authors, following the
paper~\cite{fouquet1999bipartite}. The special structure of
$(P_6,C_6)$-free graphs was observed to be compatible with certain vertex
decompositions involving $K\oplus S$ graphs, where a $K\oplus S$ stands for
a graph whose vertex set has a partition into a biclique $K$ and a stable
set $S$ (see Section~\ref{sec:background}). This, in turn, allows
$(P_6,C_6)$-free bipartite graphs to be recognized in linear
time~\cite{Takaoka:23}. The strong relationship between $P_k$-freeness and
dominating subsets in undirected graphs has been known for a long time; see
the seminal paper~\cite{bacsotuza1} by Bacs{\'o} and Tuza.  Two very recent
papers on $(P_6,C_6)$-free graphs involve $K\oplus S$ subgraphs.
In~\cite{quaddoura2024bipartite}, it is shown that a bipartite graph $G$ is
$(P_6,C_6)$-free if and only if every connected subgraph of $G$ is a
$K\oplus S$ graph.  In~\cite{Takaoka:23}, $(P_6,C_6)$-free graphs are
characterized as those with a canonical decomposition comprising only
parallel and $K\oplus S$-steps. This motivates the study of such
decompositions for un2qBMGs and their relationship to 2qBMGs. Our main
contribution in this direction is Theorem~\ref{thm:connected-typeA}
  which states that the underlying undirected graph $G$ of a 2qBMG admits a
  decomposition $K\oplus S$. Moreover, if $\dG$ is a 2qBMG satisfying a
  mild condition $(*)$, then Corollary~\ref{cor:consistent_orientation}
     shows the existence of an induced acyclic-oriented
  subgraph $\overrightarrow{K}$ of $\dG$ such that the underlying
  undirected graph $K$ of $\overrightarrow{K}$ is a dominating biclique of
  the underlying undirected graph $G$ of $\dG$.

\section{Background}\label{sec:background}
Our notation and terminology for undirected graphs are standard.  Let $G$
be an undirected graph with vertex set $V(G)$ and edge set $E(G)$.  A
\emph{$P_n$ path-graph} is an undirected graph with
$V(G)=v_1,v_2,\ldots,v_n$ and edges $\{v_iv_{i+1}\mid 1\leq i<n\}$.  An
undirected graph is \emph{$P_n$-free} if it has no induced subgraph
isomorphic to a $P_n$ path-graph; equivalently, we say that $P_n$ is a
forbidden subgraph for $G$. A $C_n$ \emph{cycle-graph} is an undirected
graph with $V(G)=v_1,v_2,\ldots,v_n$ and any edge in $E(G)$ is either
$v_iv_{i+1}$ for some $1\le i \le n-1$, or $v_nv_1$. For $n$ even, a $C_n$
cycle-graph is a bipartite graph with (uniquely determined) bipartition
classes $\{v_1,v_3,\ldots\}$ and $\{v_2,v_4, \ldots\}$. Bipartite
cycle-graphs can only exist for even $n$. An undirected graph is
$C_n$-\emph{free} if it has no induced subgraph isomorphic to an $C_n$
cycle-graph; equivalently, we say that $C_n$ is a forbidden subgraph for
$G$. 
Graph properties preserved by induced subgraphs, i.e., hereditary properties, admit a characterization by means of forbidden induced
subgraphs (even though there might be infinitely many forbidden induced
subgraphs).

A subset $D \subseteq V(G)$ is called a \emph{dominating set} of an
undirected graph $G$ if every vertex $u \in V(G)\setminus D$ has a neighbor
$v \in D$, i.e., $uv \in E(G)$.  In this case, we also say that $D$
\emph{dominates} $G$. A \emph{stable set} (also called an \emph{independent
set}) in a graph $G$ is a subset of vertices no two of which are adjacent. A \emph{biclique} of $G$ is an induced subgraph of $G$ which is isomorphic to a complete bipartite graph.  
An undirected bipartite graph $G$ is said to be of type $K \oplus S$ if either $G$ is
degenerate (that is, it contains an isolated vertex), or if the vertex set admits a partition $V(G) = K \cup S$ where $K$ is a biclique and $S$
is an independent set;
see~\cite{fouquet1999bipartite,quaddoura2024bipartite}.

For a directed graph $\dG$ with vertex set $V(\dG)$ and edge set $E(\dG)$,
we adopt the notation and terminology
of~\cite{korchmaros2023quasi}. Specifically, for $v \in V(\dG)$, we write
$N^+(v)$ and $N^-(v)$ for the sets of out-neighbours and in-neighbours
of~$v$, respectively. Two vertices $v,w$ of a digraph are \emph{equivalent} if $N^+(v)=N^+(w)$ and $N^-(v)=N^-(w).$
We refer to the two edges of a $2$-cycle as a pair of symmetric edges.

A digraph $\dG$ is \emph{oriented} if $uv \in E(\dG)$ implies
$vu \notin E(\dG)$ for all $u,v \in V(\dG)$. An oriented digraph admits a
\emph{topological vertex ordering} if its vertices can be labeled $v_1,
v_2, \ldots$ such that whenever $v_i v_j \in E(\dG)$ we have $i < j$. A
necessary and sufficient condition for the existence of a topological
ordering is that $\dG$ is \emph{acyclic}, i.e., that it contains no
directed cycle. The \emph{underlying undirected graph}
of a digraph $\dG$ is the undirected graph
$G$ on the same vertex set, where the edge set is
$\bigl\{uv: uv \in E(\dG) \text{ or } vu \in E(\dG) \bigr\}.$ A digraph $\dG$ is \emph{bipartite} if its underlying undirected graph $G$ is bipartite.

An \emph{orientation} $\overrightarrow{\Gamma}$ of a digraph $\dG$ is the digraph obtained from $\dG$ by removing exactly one edge from each pair of symmetric edges; in particular, if $\dG$ has no pair of symmetric edges, then $\dG$ has just one orientation, which coincides with $\dG$. Then we say that an orientation $\overrightarrow{\Gamma}$ of $\dG$ is \emph{consistent} if it preserves the equivalence between any two vertices, i.e., any two vertices of such digraphs are equivalent in $\overrightarrow{\Gamma}$ if and only if they are equivalent in $\dG$.

An oriented bipartite digraph is called a
\emph{bitournament} if for any two vertices $u,v \in V(\dG)$ with different
colors, either $uv \in E(\dG)$ or $vu \in E(\dG)$. 
Every bi-transitive bitournament is acyclic-oriented; see~\cite[Theorem 2.5]{DGGS}.

A directed graph is
\emph{connected} if its underlying undirected graph is connected.  A
directed graph is said to be $P_n$-\emph{free} or $C_n$-\emph{free} if its 
underlying undirected graph is $P_n$-free or $C_n$-free, respectively; equivalently,
we say that $C_n$ and $P_n$ are forbidden subgraphs for $\dG$.

A \emph{tree} is an undirected, connected graph that does not contain cycles. A vertex $v$ of a tree is a \emph{leaf} if $v$ has degree $1$,  otherwise $v$ is an \emph{inner} vertex.  
Let $T$ be a tree with leaf set $L:= L(T)\subset V(T)$. 
A \emph{leaf coloring} of $T$ is a map $\sigma\colon L\to S$ where $S$ is a non-empty set of \emph{colors}. A vertex $v$ of $T$ is a \emph{child} of $u$ if $uv\in E(T)$. If $v$ is a child of $u$, then $u$ is the (unique) \emph{parent} of $v$. If $v$ has no parent, then it is a \emph{root}. A tree is \emph{rooted} if it has a unique root $\rho$. A tree $T$ is \emph{phylogenetic} if every vertex of $T$ is either a leaf or it has at least two children. All trees appearing in this contribution are assumed to be rooted and phylogenetic. A vertex $u\in V(T)$ is an \emph{ancestor} of $v\in V(T)$ if $u$ lies on the path from $\rho$ to $v$. In this case, we write $v\preceq_T u$. Then the vertex set of $T$ w.r.t. relation $\preceq_T$ is a partially ordered set. If $uv\in E(T)$, then we write $v\prec_T u$. The \emph{least common ancestor} $\lca_{T}(A)$ is the unique $\preceq_T$-smallest vertex that is
an ancestor of all vertices in $A\subseteq V$. For brevity, we write
$\lca_{T}(x,y)$ instead of $\lca_{T}(\{x,y\})$. 
Let $(T,\sigma)$ be a leaf-colored tree with vertex set $V$, leaf set $L$ and $\sigma(L) \subseteq S$.

A leaf $y\in L(T)$ is a
  \emph{best match} of the leaf $x\in L(T)$ if $\sigma(x)\neq\sigma(y)$ and
  $\lca(x,y)\preceq_T \lca(x,y')$ holds for all leaves $y'$ of color
  $\sigma(y')=\sigma(y)$.

 A \emph{truncation map}
  $u_T\colon L\times S\to V$ assigns to every leaf $x\in L$ and color
  $s\in S$ a vertex of $T$ such that $u_T(x,s)$ lies along the unique path
  from $\rho_T$ to $x$ and $u_T(x,\sigma(x))=x$. A leaf $y \in L$ with
  color $\sigma(y)$ is a \emph{quasi-best match for $x\in L$}
  (w.r.t. $(T,\sigma)$ and $u_T$) if both conditions (i) and (ii) are
  satisfied:
  \begin{itemize}
  \item[(i)] $y$ is a \emph{best match} of $x$ in $(T,\sigma)$.
  \item[(ii)] $\lca_T(x,y) \preceq u_T(x,\sigma(y))$.
  \end{itemize}
  The vertex-colored digraph on the vertex set $L$
  whose edges are defined by the quasi-best matches is the \emph{quasi-best
  match graph} (qBMG) of $(T,\sigma,u_T)$.
 If $u_T(x,\sigma)=\rho$ for every $x\in L$, then the associated qBMG is a \emph{best match graph} (BMG). 
  
\begin{definition}
A vertex-colored digraph $(\dG,\sigma)$ with vertex set $L$ is a
\emph{colored quasi-best match graph} (qBMG) if there is a
leaf-colored tree $(T,\sigma)$ equipped with truncation map $u_T$ on
$(T,\sigma)$ such that $(\dG,\sigma) = \qBMG(T,\sigma,u_T)$.
\end{definition}

The two-colored qBMGs (2qBMGs) are characterized in terms of directed bipartite graphs; see Proposition \ref{pro23112025}. 
The vertex-coloring determines the bipartition of the vertex set, and we will refer to the bipartition
classes of a 2qBMG as color classes. A 2qBMG is {\emph{degenerate}} if it
has an isolated vertex. We stress that a non-degenerate 2qBMG may be
{\emph{trivial}}, as it may be the union of pairwise disjoint edges and $2$-cycles. If this is the case, then (N1), (N2), and (N3) trivially
hold in the sense that none of the pre-conditions in (N1), (N2), and (N3)
is satisfied.

\section{Forbidden induced paths of underlying undirected 2qBMGs}
\label{secindu}
Let $(\dG,\sigma)$ be a qBMG. The underlying undirected graph of a qBMG $(\dG,\sigma)$ is named \emph{unqBMG} and denoted by $(G,\sigma)$. In particular, un2qBMGs are bipartite graphs since 2qBMGs are bipartite, and the bipartition of the vertex set $V(G)$ coincides with the vertex-coloring $\sigma$. Recall that 2qBMGs form a hereditary family of directed graphs and are characterized by forbidden subgraphs~\cite{korchmaros2023quasi}.  We start the section by showing that un2qBMGs also form a hereditary family of graphs. This motivates us to look for forbidden configurations for the family of
un2qBMGs. Therefore, we continue the section by proving that un2qBMGs are $P_6$-free and chordal bipartite.

\begin{proposition}\label{prop:un2qbmg_hereditary}
  Every induced subgraph of an un2qBMG is an un2qBMG, i.e., the underlying undirected graphs of 2-qMBGs form a hereditary graph class.
\end{proposition}
\begin{proof}
  Let $G$ be the underlying undirected graph of a 2qBMG $\dG$. Observe that the induced subgraph $G'$ of $G$ by a subset of vertices $V'$ with $V'\subseteq V(G)$ is the underlying undirected graph of the induced subgraph $\dG'$ of $\dG$ by $V'$. From~\cite[Corollary 3.6]{korchmaros2023quasi} $G'$ is a 2qBMG, hence $G'$ is a un2qBMG.
 \end{proof}
The following theorem is the main result of this contribution.

\begin{theorem}\label{thm:P6free}
  Every un2qBMG is $P_6$-free and $C_6$-free.
\end{theorem}
\begin{proof}
Let $\overrightarrow{G}$ denote a 2qBMG with at least six vertices. Assume, on the contrary, that its underlying undirected graph $G$ has an induced
subgraph $G_6$ on six vertices $v_1,v_2,v_3,v_4,v_5,v_6$ such that
$v_1v_2v_3v_4v_5v_6$ is a $P_6$ path-graph. Then $G_6$ contains no edge
other than $v_iv_{i+1}$ for $i=1,\ldots,5$. Four cases arise according to
the possible patterns of the neighborhood of $v_2$ in $\dG$.

(i): $v_1v_2,v_3v_2\in E(\dG)$. Then $v_3v_4\in E(\dG),$ otherwise
$[v_4,v_3,v_2,v_1]$ is a (N1)-configuration.  If $v_4v_5\in E(\dG)$ then
$v_6v_5\in E(\dG)$, otherwise $v_3v_4v_5v_6$ violates (N2). On the other
hand, $v_4v_5,v_6v_5\in E(\dG)$ violates (N1), as $[v_3,v_4,v_5,v_6]$ is a
(N1)-configuration. Hence $v_5v_4\in E(\dG)$. Similarly, if $v_6v_5\in
E(\dG)$ then $[v_6,v_5,v_4,v_3]$ is a (N1)-configuration. 
Therefore, $v_1v_2,v_3v_2,v_3v_4,v_5v_4,v_5v_6\in E(\dG)$, as shown in
Figure~\ref{fig:p5}(b). But then $\{v_2,v_3,v_4,v_5,v_6\}$ 
violates (N3) with respect to $v_3$ and $v_5$.

(ii): $v_1v_2,v_2v_3\in E(\dG)$. Then $v_4v_3\in E(\dG)$, otherwise
$v_1v_2v_3v_4$ violates (N2). On the other hand, $v_4v_3\in E(\dG)$ yields
that $[v_1,v_2,v_3,v_4]$ is a (N1)-configuration.

(iii): $v_2v_1,v_2v_3\in E(\dG)$. Suppose $v_4v_3\in E(\dG)$. Then either $[v_5,v_4,v_3,v_2]$ is a (N1)-configuration, or (N3) is
  violated by $\{v_2,v_4\}$, depending on whether $v_5v_4\in E(\dG)$,
    or $v_4v_5\in E(\dG)$.  Hence $v_3v_4\in E(\dG)$.  For $v_5v_4\in
  E(\dG)$, $[v_2,v_3,v_4,v_5]$ is a (N1)-configuration, while for
  $v_4v_5\in E(\dG)$, $\{v_2,v_3,v_4,v_5\}$ violates (N2).

(iv): $v_2v_1,v_3v_2\in E(\dG)$. Then $v_3v_4\in E(\dG)$, otherwise
$v_4v_3v_2v_1$ violates (N2). If $v_5v_4\in E(\dG)$, then either
  $\{v_3,v_5\}$ violates (N3), or $[v_6,v_5,v_4,v_3]$ is an
  (N1)-configuration, according as $v_5v_6\in E(\dG)$, or $v_6v_5\in
  E(\dG)$. Hence $v_4v_5\in E(\dG)$. If $v_6v_5\in E(\dG)$ then $[v_3,v_4,v_5,v_6]$ is an (N1)-configuration, a contradiction. On the other hand, if $v_5v_6\in E(\dG)$, $v_3v_4v_5v_6$ violates (N2).
  
In conclusion, $G$ is $P_6$-free because none of the four cases applies.

Assume that $G$ has an induced subgraph $G_6$ on six vertices
$v_1,v_2,v_3,v_4,$ $v_5,v_6$ such that $v_1v_2v_3v_4v_5v_6v_1$ is a $C_6$,
i.e., a hexagon graph. Since the arguments above that rule out a $P_6$
graph involve neither $v_1v_6$ nor $v_6v_1$, they also imply that $G$ is
$C_6$-free.
\end{proof}

Notice that a $(P_6,C_6)$-free bipartite graph $G$ is chordal bipartite. In fact, it cannot
contain an induced $C_6$ (by definition) and no induced $C_k$ with $k>7$, since each such induced cycle contains an induced  $P_6$. 
\begin{corollary}\label{cor:chordal_bip}
The class of underlying undirected graphs of 2qBMGs is contained in the class of $P_6$-free
 chordal bipartite graphs.
\end{corollary}
In view of Corollary~\ref{cor:chordal_bip}, it is natural to ask
  whether the class of un2qBMGs is strictly contained in the class of
  $(P_6,C_6)$-free graphs.  The next theorem shows that this is indeed the
  case.
\begin{theorem}\label{thm:extra_forbidden}
  The Sunlet$_4$ graph in Figure~\ref{fig:sunlet4+} is a forbidden subgraph for the
  class of un2qBMGs.
\begin{figure}[ht]
  \centering
\begin{tikzpicture}[x=1cm, y=1cm, rotate=45,
    fulldot/.style={circle, fill=black, inner sep=2pt},
    emptydot/.style={circle, draw=black, fill=white, inner sep=2pt}
  ]
  \node[label=above:{$v_3$},fulldot]   (v1) at (0:1)    {};
  \node[label=above:{$v_2$},emptydot]  (v2) at (90:1)   {};
  \node[label=below:{$v_5$},fulldot]   (v3) at (180:1)  {};
  \node[label=below:{$v_8$},emptydot] (v4) at (270:1)  {};
  
  \node[label=above:{$v_4$},emptydot] (u1) at (0:2)   {};
  \node[label=above:{$v_1$},fulldot]   (u2) at (90:2)  {};
  \node[label=below:{$v_6$},emptydot]  (u3) at (180:2) {};
  \node[label=below:{$v_7$},fulldot]  (u4) at (270:2) {};
  
  
  \draw (v1) -- (v2) -- (v3) -- (v4) -- (v1);
  
  \draw (v1) -- (u1);
  \draw (v2) -- (u2);
  \draw (v3) -- (u3);
  \draw (v4) -- (u4);
  
\end{tikzpicture}
\caption{A $(P_6,C_6)$-free graph that is not an un2qBMG.}
  \label{fig:sunlet4+}
\end{figure}
\end{theorem}
\begin{proof}
It is straightforward to check that the graph $G$ in
Figure~\ref{fig:sunlet4+} is $(P_6,C_6)$-free. Now assume that $G$ is a
un2qBMG and let $\dG$ denote a 2qBMG whose underlying undirected graph is
$G$. Two cases arise according to the possible orientations of the edge
$v_1v_2$.
\begin{itemize}
\item[(i)] $v_1v_2\in E(\dG)$. Then $v_3v_2\in E(\dG)$, otherwise either
  $[v_1,v_2,v_3,v_4]$ is a (N1)-configuration or $v_1v_2v_3v_4$ violates
  (N2), depending on whether $v_4v_3\in E(\dG)$ or $v_3v_4\in
  E(\dG)$. A similar argument on $\{v_1,v_2,v_5,v_6\}$ implies $v_5v_2\in
  E(\dG)$. Furthermore, $v_5v_6\in E(\dG)$, otherwise $[v_6,v_5,v_2,v_1]$
  is a (N1)-configuration. A similar argument, by symmetry, implies that $v_3v_4\in E(\dG)$. But then $\{v_2,v_3,v_4,v_5,v_6\}$ violates (N3) with respect to $v_3$ and $v_5$.
  \item[(ii)] $v_2v_1\in E(\dG)$. Then by case (i) and by symmetry, one can assume (more generally) that $v_3v_4,v_5v_6,v_8v_7\in E(\dG)$. W.l.g. by symmetry, assume that $v_2v_3\in E(\dG)$. Thus $v_8v_3\in E(\dG)$, otherwise $v_2v_3v_8v_7$ violates (N2). But then $\{v_1,v_2,v_3,v_8,v_7\}$ violates (N3) with respect to $v_2$ and $v_8$.  
\end{itemize}
In conclusion, $G$ is not un2qBMG as neither case (i) nor case (ii) occurs.
Since un2qBMGs form a hereditary graph class, $G$ is, in particular, a forbidden induced subgraph for un2qBMGs.
\end{proof}
The following result is an immediate consequence of
Theorem~\ref{thm:extra_forbidden} and Corollary~\ref{cor:chordal_bip}.
\begin{corollary}
\label{cor:strictly_P6_C6}
The underlying undirected graphs of 2qBMGs are strictly contained in the class of
$P_6$-free chordal bipartite graphs.
\end{corollary}

In the remainder of this section, we establish that un2qBMGs may contain an induced $P_4$ or $P_5$, as well as a $C_4$. The simple but tedious proofs, which proceed case by case, are omitted for brevity. For a proof of Proposition~\ref{prop:P5}, the reader is referred to the proof of~\cite[Theorem 3.2]{AK2024ITAT}.

\begin{proposition}\label{prop:P5}
  The six non-isomorphic 2qBMGs in Figure~\ref{fig:p5}(a) are
  exactly the non-isomorphic 2qBMGs on five vertices whose underlying
  undirected graph is a $P_5$ path-graph.
  \begin{figure}
\begin{tikzpicture}[>={Stealth[bend]}, x=1cm, y=1cm]
    \tikzset{filled/.style={fill,circle,inner sep=2pt}} 
    \tikzset{open/.style={draw,circle,inner sep=2pt}}   
    \begin{scope}[scale=1]
        \node[label=right:{$v_4$},filled] (4) at (1, 0) {};
        \node[label=right:{$v_5$}, open] (5) at (0.309, 0.951) {};
        \node[label=left:{$v_1$},open] (1) at (-0.809, 0.588) {};
        \node[label=left:{$v_2$},filled] (2) at (-0.809, -0.588) {};
        \node[label=right:{$v_3$},open] (3) at (0.309, -0.951) {};
    \end{scope}
    \begin{scope}[every edge/.style={draw=black}, scale=0.4]
        \draw [->] (1) -- (2);
        \draw [<-] (2) -- (3);
        \draw [->] (3) -- (4);
        \draw [->] (4) -- (5);
    \end{scope}
\node at (-1,1.9) {\large{(a)}}; 
\end{tikzpicture}
\begin{tikzpicture}[>={Stealth[bend]}, x=1cm, y=1cm]
    \tikzset{filled/.style={fill,circle,inner sep=2pt}} 
    \tikzset{open/.style={draw,circle,inner sep=2pt}}   
    \begin{scope}[scale=1]
        \node[label=right:{$v_4$},filled] (4) at (1, 0) {};
        \node[label=right:{$v_5$}, open] (5) at (0.309, 0.951) {};
        \node[label=left:{$v_1$},open] (1) at (-0.809, 0.588) {};
        \node[label=left:{$v_2$},filled] (2) at (-0.809, -0.588) {};
        \node[label=right:{$v_3$},open] (3) at (0.309, -0.951) {};
    \end{scope}
    \begin{scope}[every edge/.style={draw=black}, scale=0.4]
        \draw [->] (1) -- (2);
        \draw [<-] (2) -- (3);
        \draw [->] (3) -- (4);
        \draw [<-] (4) -- (5);
    \end{scope}
\end{tikzpicture}
\begin{tikzpicture}[>={Stealth[bend]}, x=1cm, y=1cm]
    \tikzset{filled/.style={fill,circle,inner sep=2pt}} 
    \tikzset{open/.style={draw,circle,inner sep=2pt}}   
    \begin{scope}[scale=1]
        \node[label=right:{$v_4$},filled] (4) at (1, 0) {};
        \node[label=right:{$v_5$}, open] (5) at (0.309, 0.951) {};
        \node[label=left:{$v_1$},open] (1) at (-0.809, 0.588) {};
        \node[label=left:{$v_2$},filled] (2) at (-0.809, -0.588) {};
        \node[label=right:{$v_3$},open] (3) at (0.309, -0.951) {};
    \end{scope}
    \begin{scope}[every edge/.style={draw=black}, scale=0.4]
        \draw [<-] (1) -- (2);
        \draw [<-] (2) -- (3);
        \draw [->] (3) -- (4);
        \draw [->] (4) -- (5);
    \end{scope}
\end{tikzpicture}
\hspace{-0.5cm} 
\begin{tikzpicture}[>={Stealth[bend]},x=1cm,y=1cm,bullet/.style={fill,circle,inner sep=3pt}]
    \tikzset{filled/.style={fill,circle,inner sep=2pt}} 
    \tikzset{open/.style={draw,circle,inner sep=2pt}}   
\begin{scope}[scale=0.7]
    \node[label=below:{$v_1$}, open] (v1) at (0,0) {};
   \node[label=above:{ $v_2$}, filled] (v2) at (0,2) {};
   \node[label=below:{ $v_3$}, open] (v3) at (2,0) {};
   \node[label=above:{ $v_4$}, filled] (v4) at (2,2) {};
   \node[label=above:{ $v_6$}, filled] (v6) at (4,2) {};
   \node[label=below:{ $v_5$}, open] (v5) at (4,0) {};
 \end{scope}
\begin{scope}[ every edge/.style={draw=black},scale=1.5]
\draw [-] (v1) edge [->, thick] (v2);
\draw [-] (v3) edge [->, thick] (v2);
\draw [-] (v3) edge [->, thick] (v4);
\draw [-] (v5) edge [->, thick] (v4);
\draw [-] (v5) edge [->, thick] (v6);
\end{scope}
\node at (0,2.5) {\large{(b)}}; 
\end{tikzpicture}
\caption{(a) 2qBMGs with 5-path-graph un2qBMG:
  $\overrightarrow{P}_5^{(a)}$ graph (left), $\overrightarrow{P}_5^{(b)}$
  graph (middle), and $\overrightarrow{P}_5^{(c)}$ graph (right). The
  vertex set of
  $\overrightarrow{P}_5^{(ab)},\overrightarrow{P}_5^{(ac)},\overrightarrow{P}_5^{(abc)}$
  coincides with $V(\overrightarrow{P}_5^{(a)})$. The edge sets are
  $E(\overrightarrow{P}_5^{(ab)}):=E(\overrightarrow{P}_5^{(a)})\cup
  E(\overrightarrow{P}_5^{(b)}),
  E(\overrightarrow{P}_5^{(ac)}):=E(\overrightarrow{P}_5^{(a)})\cup
  E(\overrightarrow{P}_5^{(c)}),$ and $
  E(\overrightarrow{P}_5^{(abc)}):=E(\overrightarrow{P}_5^{(a)})\cup
  E(\overrightarrow{P}_5^{(b)})\cup E(\overrightarrow{P}_5^{(c)})$. (b)
  Case (i) of the proof of Theorem~\ref{thm:P6free}: an orientation of
  $\protect\dG$.}
\label{fig:p5}
\end{figure}
\end{proposition}

\begin{proposition}\label{thm:P4-cases}
  The four non-isomorphic 2qBMGs in Figure~\ref{fig:c4}(b) are
  exactly those with $P_4$ as their underlying undirected graph.
\end{proposition}

\begin{proposition} \label{thm:C4-cases}
  The ten graphs in Figure~\ref{fig:c4}(a) are exactly the non-isomorphic
  2qBMGs on four vertices whose underlying undirected graph is a $C_4$
  cyle-graph.

\begin{figure}[ht!]
    \centering
\begin{tikzpicture}[>={Stealth[bend]},x=1cm,y=1cm]
    \tikzset{filled/.style={fill,circle,inner sep=2pt}} 
    \tikzset{open/.style={draw,circle,inner sep=2pt}}   
\begin{scope}[scale=0.4]
   \node[filled] (b) at (4,0) {};
   \node[filled] (c) at (0,4) {};
   \node[open] (a) at (0,0) {};
   \node[open] (d) at (4,4) {};
 \end{scope}
\begin{scope}[ every edge/.style={draw=black} ,scale=0.4]
\draw [-] (a) edge [->, thick] (b);
\draw [-] (c) edge [<-, thick] (a);
\draw [-] (b) edge [<-, thick] (d);
\draw [-] (c) edge [->, thick] (d);
\draw [<-,gray,dashed,thick] (a.north) to [out=30,in=330] (c.south);
\draw [<-,gray,dashed,thick] (d.south) to [out=120,in=150] (b.north);
\draw [->,gray,dashed,thick] (d.north) to [out=150,in=30] (c.north);
\end{scope}
\node at (-0.5,2) {\large{(a)}}; 
\end{tikzpicture}
\begin{tikzpicture}[>={Stealth[bend]},x=1cm,y=1cm]
    \tikzset{filled/.style={fill,circle,inner sep=2pt}} 
    \tikzset{open/.style={draw,circle,inner sep=2pt}}   
\begin{scope}[scale=0.4]
   \node[open] (b) at (4,0) {};
   \node[open] (c) at (0,4) {};
   \node[filled] (a) at (0,0) {};
   \node[filled] (d) at (4,4) {};
 \end{scope}
\begin{scope}[ every edge/.style={draw=black} ,scale=0.4]
\draw [-] (b) edge [<-, thick] (a);
\draw [-] (a) edge [<-, thick] (c);
\draw [-] (b) edge [<-, thick] (d);
\draw [-] (c) edge [->, thick] (d);
\draw [->,gray,dashed,thick] (d.north) to [out=150,in=30] (c.north);
\draw [<-,gray,dashed,thick] (d.south) to [out=120,in=150] (b.north);
\end{scope}
\end{tikzpicture}
\begin{tikzpicture}[>={Stealth[bend]},x=1cm,y=1cm]
    \tikzset{filled/.style={fill,circle,inner sep=2pt}} 
    \tikzset{open/.style={draw,circle,inner sep=2pt}}   
\begin{scope}[scale=0.4]
   \node[open] (b) at (4,0) {};
   \node[open] (c) at (0,4) {};
   \node[filled] (a) at (0,0) {};
   \node[filled] (d) at (4,4) {};
 \end{scope}
\begin{scope}[ every edge/.style={draw=black} ,scale=0.4]
\draw [-] (b) edge [->, thick] (a);
\draw [-] (a) edge [<-, thick] (c);
\draw [-] (b) edge [<-, thick] (d);
\draw [-] (c) edge [->, thick] (d);
\draw [->,gray,dashed,thick] (d.north) to [out=150,in=30] (c.north);
\draw [<-,gray,dashed,thick] (d.south) to [out=120,in=150] (b.north);
\end{scope}
\end{tikzpicture}
\begin{tikzpicture}[>={Stealth[bend]},x=1cm,y=1cm]
    \tikzset{filled/.style={fill,circle,inner sep=2pt}} 
    \tikzset{open/.style={draw,circle,inner sep=2pt}}   
\begin{scope}[scale=0.4]
   \node[open] (b) at (4,0) {};
   \node[open] (c) at (0,4) {};
   \node[filled] (a) at (0,0) {};
   \node[filled] (d) at (4,4) {};
 \end{scope}
\begin{scope}[ every edge/.style={draw=black} ,scale=0.4]
\draw [-] (b) edge [<-, thick] (a);
\draw [->,thick] (b.north) to [out=150,in=30] (a.north);
\draw [-] (a) edge [<-, thick] (c);
\draw [<-,thick] (c.south) to [out=120,in=150] (a.north);
\draw [-] (b) edge [<-, thick] (d);
\draw [<-,thick] (d.south) to [out=120,in=150] (b.north);
\draw [-] (c) edge [<-, thick] (d);
\draw [->,thick] (c.north) to [out=30,in=150] (d.north);
\end{scope}
\end{tikzpicture}
\begin{tikzpicture}[>={Stealth[bend]},x=1cm,y=1cm]
    \tikzset{filled/.style={fill,circle,inner sep=2pt}} 
    \tikzset{open/.style={draw,circle,inner sep=2pt}}   
\begin{scope}[scale=0.4]
   \node[open] (b) at (4,0) {};
   \node[open] (c) at (0,4) {};
   \node[filled] (a) at (0,0) {};
   \node[filled] (d) at (4,4) {};
 \end{scope}
\begin{scope}[ every edge/.style={draw=black},scale=0.4]
\draw [-] (a) edge [->, thick] (c);
\draw [-] (a) edge [->, thick] (b);
\draw [-] (b) edge [->, thick] (d);
\draw [->,gray,dashed,thick] (b.north) to [out=150,in=30] (a.north);
\end{scope}
\node at (-0.5,2) {\large{(b)}}; 
\end{tikzpicture}
\begin{tikzpicture}[>={Stealth[bend]},x=1cm,y=1cm]
    \tikzset{filled/.style={fill,circle,inner sep=2pt}} 
    \tikzset{open/.style={draw,circle,inner sep=2pt}}   
\begin{scope}[scale=0.4]
  \node[open] (b) at (4,0) {};
   \node[open] (c) at (0,4) {};
   \node[filled] (a) at (0,0) {};
   \node[filled] (d) at (4,4) {};
 \end{scope}
\begin{scope}[ every edge/.style={draw=black} ,scale=0.4]
\draw [-] (b) edge [->, thick] (a);
\draw [-] (c) edge [->, thick] (a);
\draw [-] (b) edge [->, thick] (d);
\draw [->,gray,dashed,thick] (a.north) to [out=30,in=330] (c.south);
\end{scope}
\end{tikzpicture}
  \caption{2qBMGs for 4-path-graph and 4-cycle-graph un2qBMGs. Black edges
    are required. Dashed edges are optional and taken in (a) one at a time
    (middle-right) or in pairs: north-east, north-west, or east-west (left).}
    \label{fig:c4}
\end{figure}
\end{proposition}

\section{Dominating sets in 2qBMGs}

Theorem~\ref{thm:P6free}, together with previous results, gives new
insights into vertex decompositions involving dominating biclique
sets. Recall that a $K\oplus S$ decomposition of $V(G)$ is a
  partition of the vertex set into vertices from a biclique $K$ and
  vertices from a stable set $S$. The next result follows directly from
  Corollary~\ref{cor:chordal_bip} and~\cite[Lemma~7]{Takaoka:23}.
\begin{theorem}
  \label{thm:connected-typeA}
  The underlying undirected graph of any 2qBMG has a vertex decomposition into $K\oplus S$, where
  $K$ is a dominating set and $S$ is a stable set.
\end{theorem}
Note that the dominating biclique $K$ is an un2qBMG, since it arises as the underlying undirected graph of a digraph in which there is an edge from each vertex of $K$ to each vertex of opposite color of $K$.

However, other 2qBMGs may have $K$ as an underlying
undirected graph. In the following, we investigate when an orientation of $\dG$
induces such a 2qBMG.

\begin{proposition}
  \label{prop:delta_bmg}
Let $\dG$ be a 2qBMG. {Let $\overrightarrow{K}$ be an induced subgraph of $\dG$ such that the underlying undirected graph $K$ of $\overrightarrow{K}$ is a biclique, and thus $K$ is a biclique of the underlying undirected graph $G$ of $\dG$. If $\overrightarrow{K}$ has no pair of symmetric edges, the $\overrightarrow{K}$ is a 2qBMG.}
\end{proposition}
\begin{proof} 
By Proposition \ref{pro23112025}, it is sufficient to show that $\overrightarrow{K}$ satisfies (N1),(N2) and (N3). Notice that $\overrightarrow{K}$ satisfies (N1), since no biclique has an induced $P_4$. Furthermore, since $\overrightarrow{K}$ is an induced subgraph of $\dG$ and since $\overrightarrow{K}$ has no pair of symmetric edges, if $\overrightarrow{K}$ should not satisfy (N2), then $\dG$ would not satisfy (N2) as well, which would contradict that $\dG$ is a 2qBMG. A similar argument applies to (N3).  
\end{proof}

\begin{theorem}
\label{thm:odd-even}
Let $\dG$ be a 2qBMG. If $\dG$ admits a consistent orientation $\overrightarrow{\Gamma}$, then $\overrightarrow{\Gamma}$ has an induced acyclic-oriented subgraph $\overrightarrow{K}$ (which is not necessarily an induced subgraph of $\dG$), such that the underlying undirected graph $K$ of $\overrightarrow{K}$ is a dominating biclique of the underlying undirected graph $G$ of $\dG$.

\end{theorem}
\begin{proof} 
By assumption, $\dG$ admits a consistent orientation $\overrightarrow{\Gamma}$. Since $\dG$ is a 2qBMG, $\dG$ is a bi-transitive bipartite digraph. Furthermore, from~\cite[Theorem 4]{korchmaros2021structure}, one has that every (possible) orientation of a bi-transitive bipartite digraph is acyclic-oriented. It follows that $\overrightarrow{\Gamma}$ is acyclic-oriented. On the other hand, from Theorem~\ref{thm:connected-typeA}, the underlying undirected graph $G$ of $\dG$ contains an induced subgraph $K$ that is a dominating biclique of $G$. Let $\overrightarrow{K}$ be the induced subgraph of $\overrightarrow{\Gamma}$ on the vertex-set of $K$. $\overrightarrow{K}$ has the desired property by construction. 
\end{proof}
The existence of a consistent orientation of a 2qBMG is ensured under a mild condition.
\begin{lemma}
\label{lem:consistent_orientation} 
Let $\dG$ be a 2qBMG satisfying the following condition:
\begin{itemize}
\item[($*$)] no two (or more) pairs of symmetric edges of $\dG$ have a common endpoint.
\end{itemize}
Then $\dG$ admits a consistent orientation $\overrightarrow{\Gamma}$.
\end{lemma}
\begin{proof} 
Clearly $\dG$ admits at least one orientation. From \cite[Lemma 2.2]{korchmaros2021structure}, if some orientation is not consistent, then two (or more) pairs of symmetric edges of $\dG$ have a common vertex. Thus, since $\dG$ satisfies condition $(*)$, $\dG$ admits a consistent orientation. 
\end{proof}
As a corollary of Theorem~\ref{thm:odd-even} and Lemma~\ref{lem:consistent_orientation}, we have the following result.
\begin{corollary}
\label{cor:consistent_orientation} 
Let $\dG$ be a 2qBMG 
satisfying condition $(*)$. Then there exists a consistent orientation $\overrightarrow{\Gamma}$ of $\dG$ with an induced acyclic-oriented subgraph $\overrightarrow{K}$ (which is not necessarily an induced subgraph of $\dG$) such that the 
underlying undirected graph $K$ of $\overrightarrow{K}$ is a dominating biclique of the underlying undirected graph $G$ of $\dG$. 
\end{corollary}
It may be noticed that acyclic-oriented digraphs have been
  characterized as odd-even digraphs; see~\cite[Theorem 3.4]{DGGS}.
  Theorem~\ref{thm:extra_forbidden} poses the following problem: How can
  the class of the underlying undirected 2qBMGs be characterized in terms
  of forbidden subgraphs. More precisely, does the class of
  $(P_6,C_6,\text{Sunlet}_4)$-free graphs coincide with the class of underlying undirected
  graphs of 2qBMGs, or are additional forbidden subgraphs still to be identified?

\section*{Acknowledgements}
\noindent
The authors thank Marc Hellmuth for insightful discussions and comments, and the referees for their
valuable comments and suggestions to improve the paper’s quality.
Funded by the German Research Foundation (DFG, STA 850/49-1 432974470).

\bibliographystyle{plainurl}
\bibliography{refs}

\end{document}